\documentclass{amsart}

\usepackage{amssymb}

\usepackage{graphicx,enumerate}

\newtheorem{theorem}{Theorem}[section]
\newtheorem{lemma}[theorem]{Lemma}
\newtheorem{proposition}[theorem]{Proposition}
\newtheorem{question}[theorem]{Question}
\newtheorem{corollary}[theorem]{Corollary}
\newtheorem{claim}[theorem]{Claim}

\theoremstyle{definition}
\newtheorem{definition}[theorem]{Definition}

\theoremstyle{remark}
\newtheorem{remark}[theorem]{Remark}

\numberwithin{equation}{section}

\begin{document}

\title{Definable types in algebraically closed valued fields}


\author{Pablo Cubides Kovacsics}
\address{Laboratoire Paul Painlev\'e,
Universit\'e Lille 1,
CNRS U.M.R. 8524,
59655 Villeneuve d'Ascq Cedex​,
 France}
\curraddr{}
\email{Pablo.Cubides@math.univ-lille1.fr}
\thanks{The first author was supported by the Marie Curie Initial Training Network in Mathematical Logic - MALOA
- From MAthematical LOgic to Applications, PITN-GA-2009-238381.}

\author{Fran\c{c}oise Delon}
\address{\'Equipe de Logique Math\'ematique,
Institut de Math\'ematiques de Jussieu - Paris Rive Gauche,
Universit\'e Paris Diderot,
UFR de math\'ematiques, case 7012,
75205 Paris Cedex 13,
France}
\curraddr{}
\email{delon@math.univ-paris-diderot.fr}
\thanks{}


\date{22/04/2014}

\dedicatory{}


\begin{abstract}
Marker and Steinhorn shown in \cite{markerETAL:94} that given two models $M\prec N$ of an o-minimal theory, if all 1-types over $M$ realized in $N$ are definable, then all types over $M$ realized in $N$ are definable. In this article we characterize pairs of algebraically closed valued fields satisfying the same property. Although it is true that if $M$ is an algebraically closed valued field such that all 1-types over $M$ are definable then all types over $M$ definable, we build a counterexample for the relative statement, \textit{i.e.}, we show for any $n\geq 1$ that there is a pair $M\prec N$ of algebraically closed valued fields such that all $n$-types over $M$ realized in $N$ are definable but there is an $n+1$-type over $M$ realized in $N$ which is not definable. Finally, we discuss what happens in the more general context of $C$-minimality. 
\end{abstract}

\maketitle

\section{Introduction}

Let $L$ be a first order language, $\Sigma$ an $L$-theory and $M$ a model of $\Sigma$. For a subset $A\subseteq M$ and a type $p(x)\in S(A)$, recall that $p(x)$ is \emph{definable} if for every $L$-formula $\phi(x,y)$ there is an $L(A)$-formula $\delta_\phi(y)$ such that for all $a\in A^{|y|}$, we have that $\phi(x,a)\in p$ if and only if $M\models \delta_\phi(a)$ ($x$ and $y$ denote finite tuples). Stable theories are precisely those theories for which over \emph{every} model all types are definable. However, it is possible to find models of an unstable theory over which all types are also definable. For instance, by a result of van den Dries in \cite{vandendries:86}, all types over the real field are definable. The same is true for the $p$-adics by a result of Delon in \cite{delon:89}. Marker and Steinhorn generalized van den Dries result in \cite{markerETAL:94} showing that any o-minimal structure over which all 1-types are definable has all types definable. In fact, they proved a stronger result, namely, that given two models $M\prec N$ of an o-minimal theory, if all 1-types over $M$ realized in $N$ are definable, then all types over $M$ realized in $N$ are definable. Let us introduce some notation: 

\begin{definition}
Let $\Sigma$ be an $L$-theory, $M\prec N$ two models of $\Sigma$. For $n\geq 1$, we note $T_n(M,N)$ if for every $a\in N^n$, $tp(a/M)$ is definable. We use $T_n(M)$ to denote $T_n(M,N)$ with $N$ an $|M|^+$-saturated elementary extension of $M$ or, equivalently, to say that $M$ satisfies $T_n(M,N)$ for every elementary extensions $N$.
\end{definition}

What was previously stated can be summarized as:

\begin{theorem}\label{general1}
	Let $\Sigma$ be an $L$-theory. Then
	\begin{enumerate}[(1)]
	\item $\Sigma$ is stable if and only if $\displaystyle\bigwedge_{n\geq 1} T_n(M,N)$ for all models $M\prec N$ of $\Sigma$.
	\item If $\Sigma$ is o-minimal and $M\prec N$ are two models of $\Sigma$ then $T_1(M,N)$ implies $\displaystyle\bigwedge_{n\geq 1} T_n(M,N)$.
	\end{enumerate}
\end{theorem}

Part (1) goes back to Shelah (a proof of this can be found in \cite{pillay:83}). Part (2) corresponds to the Marker-Steinhorn theorem proved in \cite{markerETAL:94} and reproved in \cite{pillay:94} by Pillay. Notice that if for all models $M\prec N$ of a theory $\Sigma$ it is true that $T_1(M,N)$ implies $T_n(M,N)$ for all $n\geq 1$, then it is true in particular that $T_1(M)$ implies $T_n(M)$ for all $n\geq 1$. As a consequence, part (2) implies that a model of an o-minimal theory over which all 1-types are definable is a model over which all types are definable. Counterexamples to possible generalizations of (2) are the following:

\begin{theorem}[Baizhanov] There is a weakly o-minimal theory $\Sigma$ and models $M\prec N$ of $\Sigma$ such that $T_1(M,N)$ but not $T_2(M,N)$.
\end{theorem}

\begin{theorem}[Chernikov-Simon]\label{simoncherni} There is a NIP theory $\Sigma$ and a model $M$ of $\Sigma$ such that $T_1(M)$ but not $T_2(M)$.
\end{theorem}

Their proofs can be found in \cite{baizhanov:06} and \cite{chernikovETAL:12} respectively. It is then natural to ask if it is still possible to find analogues of Marker-Steinhorn's theorem in other subclasses of NIP, for instance in $C$-minimal structures. 

\begin{definition} Let $C(x, y, z)$ be a ternary relation. A \emph{$C$-set} is a structure $(M,C)$
satisfying axioms (C1)-(C4):

\begin{enumerate}
\item[(C1)] $\forall xyz (C(x, y, z)\rightarrow C(x, z, y))$,
\item[(C2)] $\forall xyz (C(x, y, z)\rightarrow \neg C(y, x, z))$,
\item[(C3)] $\forall xyzw (C(x, y, z)\rightarrow(C(w, y, z)\vee C(x, w, z)))$,
\item[(C4)] $\forall xy (x \neq y \rightarrow C(x, y, y))$,
\item[(D)] 	$\forall xy (x \neq y \rightarrow \exists z(z\neq y \wedge C(x,y,z)))$.
\end{enumerate}	

If in addition $(M,C)$ satisfies axiom (D) we say it is a \emph{dense} $C$-set. A \emph{$C$-structure} is any expansion of a $C$-set. 
\end{definition}

\begin{definition} A $C$-structure $M$ is \emph{$C$-minimal} if for every elementary equivalent structure $N\equiv M$, every definable subset $D\subseteq N$ is definable by a quantifier free formula using only the $C$-predicate. A complete theory is \emph{$C$-minimal} if it has a $C$-minimal model.
\end{definition}

O-minimal and $C$-minimal structures have both differences and similarities. If a linearly ordered structure $M$ is such that any definable subset $D\subseteq M$ is a finite union of intervals and points then this is also true for any structure $N$ elementary equivalent to $M$. In contrast, there are $C$-minimal structures $M\equiv N$ such that in $M$ every definable subset $D\subseteq M$ is a Boolean combination of quantifier free $\{C\}$-formulas but this is not true in $N$. In addition, $C$-minimal theories do not need to satisfy the exchange property nor need to have prime models over sets, two properties satisfied by o-minimal theories (see \cite{macphersonETAL:96}). Nevertheless, there is a cell decomposition theorem for dense $C$-minimal structures together with a well-behaved notion of topological dimension, both introduced by Haskell and Macpherson in \cite{macphersonETAL:94} (see also \cite{cubides:14}). As previously stated, it is thus natural to ask the following questions: 

\begin{question}\label{relativea}[Solved] Let $\Sigma$ be a $C$-minimal theory and $M\prec N$ models of $\Sigma$. Is it true that $T_1(M,N)$ implies $T_n(M,N)$ for all $n\geq 1$?
\end{question}
 
\begin{question}\label{absoluea}[Open] Let $\Sigma$ be a $C$-minimal theory and $M$ a model of $\Sigma$. Is it true that $T_1(M)$ implies $T_n(M)$ for all $n\geq 1$?
\end{question}

In what follows we will give a negative answer to question \ref{relativea} by studying pairs $M\prec N$ of algebraically closed non-trivially valued fields (whose common theory is denoted by ACVF). By quantifier elimination, one can show that algebraically closed valued fields are $C$-minimal with respect to the $C$-relation defined by $C(x,y,z)\Leftrightarrow v(x-y)<v(y-z)$. In this context where the $C$-relation is given by an ultrametric, $C$-minimality is equivalent to the fact that any definable subset of the structure is a Boolean combination of open or closed balls. We prove the following theorem (all terms to be defined): 

\begin{theorem}\label{charACVF}
Let $(K\prec L,v)$ be to models of $ACVF$. The following are equivalent
\begin{enumerate}
	\item $T_n(K,L)$ holds for all $n\geq 1$.
	\item $(K\subseteq L,v)$ is a separated extension and $T_1(vK,vL)$ holds.
\end{enumerate}
\end{theorem}

Using this theorem we show : 

\begin{proposition}\label{counterexample}
	For any $n\geq 1$ there is an extension $K\prec L$ of models of $ACVF$ such that $T_n(K,L)$ holds but $T_{n+1}(K,L)$ does not.
\end{proposition}

Note that the construction of this example is suggested in \cite{baur:82}. It is also closely related to the ``necessity of geometric sorts'' for elimination of imaginaries in algebraically closed valued fields (see \cite{haskellETAL:06} 3.5). Also, recall that question \ref{absoluea} restricted to models of $ACVF$ has a positive answer (see \cite{delon:89} or theorem \ref{absolute} in the present article). The article will be divided as follows: section 1 is devoted to the proof of theorem \ref{charACVF}; in section 2, the counterexample to question \ref{relativea} is presented. 

\section{Definable types in $ACVF$}

The background used on valued fields can be found essentially in \cite{ribenboim:64}. For valued fields $(K\subseteq L, v)$ we will denote either by $K/v, L/v$ or $k,l$ their respective residue fields and by $vK, vL$ their valuation groups. All completions of ACVF are determined by the characteristic of the field and the characteristic of the residue field. They are denoted by $ACVF_{x,y}$, where $x$ is the characteristic of the field and $y$ is the characteristic of the residue field. We recall some terminology and notation:

\begin{definition} Let $(K\subseteq L,v)$ be an extension of valued fields
\begin{enumerate}
	\item A point $a\in L$ is \emph{limit }over $K$ if for any ball $B$ of $K$ containing $a$ there is a ball $B'$ in $K$ containing $a$ such that $B'\subsetneq B$.
	\item $K$ is \emph{maximal }if any family of nested non-empty balls has non-empty intersection.
	\item $K$ is \emph{definably maximal} if every definable family of nested non-empty balls has non-empty intersection.
	\item For $x\in L$ and $y_1,\ldots,y_n\in L$,
	$$I_n(x, y_1,\ldots,y_n;K):=\left\{v\left(x+\sum_{i=1}^n e_iy_i\right): e_1,\ldots,e_n\in K\right\}.$$
\end{enumerate}
\end{definition}

\begin{remark}
$K$ is maximal if and only if no extension $L\supseteq K$ contains a limit point over $K$. Clearly maximal implies definably maximal and this later property is first order definable. Every completion of $ACVF$ has a model which is a maximal valued field (for instance the Hahn field $k((G))$ with $G$ a divisible ordered abelian group and $k$ a model of $ACF_0$ is such a model of $ACVF_{0,0}$). Thus all models of $ACVF$ are definably maximal.
\end{remark}

\begin{definition}
Let $\Sigma$ be an $L$-theory, $M\prec N$ two models of $\Sigma$ and $n<\omega$. Subsets of $M^n$ of the form $\phi(M)=\{a\in M^n: N\models\phi(a)\}$ for some $L(N)$-formula $\phi(x)$ are called \emph{externally definable}. For $n\geq 1$, we let $W_n(M,N)$ state that for every $L(N)$-formula $\phi(x)$ with $|x|=n$, the set $\phi(M)$ is $M$-definable. In other terms, every externally definable set with parameters in $N$ and $n$ variables is $M$-definable.
As with $T_n(M)$, we use $W_n(M)$ to denote $W_n(M,N)$ with $N$ an $|M|^+$-saturated elementary extension of $M$ or, equivalently, to say that $M$ satisfies $W_n(M,N)$ for every elementary extensions $N$.
\end{definition}

It is a routine exercise to show that 

$$\displaystyle\bigwedge_{n\geq 1} T_n(M,N) \Leftrightarrow \bigwedge_{n\geq 1} W_n(M,N).$$

\begin{proposition}\label{S1} Let $(K\prec L,v)$ be models of $ACVF$. Then the following are equivalent:
\begin{enumerate}
	\item $W_1(K,L)$;
	\item for any ball $B$ in $L$, $B\cap K$ is $K$-definable;
	\item $T_1(vK,vL)$.
\end{enumerate}
\end{proposition}

\begin{proof} Recall that a 1-type over a model of an o-minimal theory is not definable if and only if the type defines a cut\footnote{By definition, an element $\gamma\in vL$ realizes a cut over $vK$ if there are subsets $U,V\subsetneq vK$ such that $U\cap V=\emptyset$, $U\cup V=vK$ and $\forall x\in U\forall y\in V(x<\gamma<y)$.}. We show $1\Rightarrow 3$ by contraposition, so let $\gamma\in vL$ realize a cut over $vK$. Then it is easy to see that the set $\{x\in K:v(x)>\gamma \}$ is not a Boolean combination of $K$-definable balls, \textit{i.e.}., it is not definable in $K$. For $3\Rightarrow 2$, given a ball $B$ in $L$, either its radius (\emph{i.e.} valuational radius) $\gamma$ is $\gamma\in vK$, or $\gamma$ realizes a type of the form $\beta^{\pm}$ with $\beta\in vK$, or $\gamma<vK$, or $vK<\gamma$. We can assume that the ball $B$ is centered at $a\in K$, for if not, then $B\cap K=\emptyset$ which is $K$-definable. Notice this deals with the case $vK<\gamma$. If $\gamma\in vK$, there is nothing to show. Suppose that $\gamma<vK$, then $B\cap K=K$, which is again $K$-definable. Suppose that $\gamma$ realizes a type of the form $\beta^+$ with $\beta\in vK$. Then the $K$-definable open ball $B'$ of radius $\beta$ centered at $a$ satisfies that $B'=B\cap K$. Suppose that $\gamma$ realizes a type of the form $\beta^-$ with $\beta\in vK$. Then the $K$-definable closed ball $B'$ of radius $\beta$ centered at $a$ satisfies that $B'=B\cap K$. Finally, ($2\Rightarrow 1)$ holds by $C$-minimality. \end{proof}

Let $(K,v)$ be a model of $ACVF$. The space of 1-types over $K$ has three kinds of non-realized types: residual types, valuational types and limit types. Let $p$ be a non-realized 1-type over $K$ and let $a$ be a realization of $p$ in some elementary extension $(L, v)$ of $(K,v)$. The type $p$ is \emph{residual} if $K/v\subsetneq K(a)/v$, or equivalently for some $k_1,k_2\in K$, $v(k_1a+k_2)=0$ and $(k_1a+k_2)/v \notin K/v$ (in this case we have that $vK(a)=vK$). The type $p$ is valuational, if $vK\subsetneq vK(a)$, or equivalently if for some $k\in K$, $v(a-h)\notin vK$ (in this case $K(a)/v=K/v$). Finally, $p$ is a \emph{limit type} if $K(a)$ is an immediate extension of $K$, or equivalently if $a$ is limit over $K$. Notice that a residual type is always definable and valuational types are not definable exactly when they determine a cut in $vK$. 

\begin{corollary}\label{T1W1} Let $(K\prec L,v)$ be models of $ACVF$. Then $T_1(M,N)$ implies $W_1(M,N)$.
\end{corollary}

\begin{proof} It is easy to see that $T_1(K,L)$ implies $T_1(vK,vL)$. Thus $W_1(K,L)$ follows by \ref{S1}. \end{proof}

\begin{proposition}\label{limitnotdef} Let $p$ be a limit type over $K$. Then $p$ is not definable. 
\end{proposition}
\begin{proof} Suppose that $a\in L$ is a limit point over $K$. Consider $D:=\{(y,z)\in K^2:v(a-y)\geq v(z)\}$. If $tp(a/K)$ was definable, the family $D_z:=\{y\in K: (y,z)\in D\}$ indexed by all $z\in K$ such that $D_z$ is not empty, would be a definable family of nested balls with empty intersection, contradicting that all algebraically closed valued fields are definably maximal. \end{proof}

\begin{corollary}\label{T1}
$T_1(K,L)$ if and only if $T_1(vK,vL)$ and $L$ does not contain any limit point over $K$. 	
\end{corollary}

\begin{proof} Suppose $T_1(K,L)$. We clearly have $T_1(vK,vL)$ and by Proposition (\ref{limitnotdef}) $L$ does not contains a limit point over $K$. For the right-to-left implication, given that $L$ does not contain limit points, all 1-types realized in $L$ are either residual or valuational. Moreover, since $T_1(vK,vL)$ holds, all valuational types realized in $L$ are definable. Therefore, $T_1(K,L)$ holds. \end{proof}

\begin{corollary}
	$W_1(K,L)$ does not imply $T_1(K,L)$ or $W_2(K,L)$.
\end{corollary}

\begin{proof} Take $K$ not maximal and $L$ a proper immediate extension. Therefore, since $vK=vL$ we have that $T_1(vK,vL)$ and by \ref{S1}, $W_1(K,L)$. However, for any $a\in L\setminus K$, $a$ is limit over $K$ and hence both $tp(a/K)$ and the set $D:=\{(y,z)\in K^2:v(a-y)\geq v(z)\}$ are not definable. \end{proof}

We present some known results on divisible ordered abelian groups that will be later used (see chapter 2 of \cite{lewenberg:95}). Let $DOAG$ be the theory of non-trivial divisible ordered abelian groups. This theory is complete and model-complete. If $G$ and $H$ are ordered groups, we use $G\trianglelefteq H$ to express that $G$ is a \emph{convex subgroup} of $H$. 

\begin{proposition}\label{doag} Let $G\prec H$ be models of $DOAG$. The following are equivalent:
\begin{enumerate}
	\item $T_1(G,H)$;
	\item  $T_n(G,H)$ for all $n\geq 1$;
	\item there are convex subgroups $C_1\trianglelefteq C_2 \trianglelefteq H$ such that $C_2=G\times C_1$ with the lexicographic order;
	\item for all $x\in H$, if there is $y\in G$ such that $|x|<|y|$ then there is some $z\in G$ such that $|x-z|<G^{>0}$.
\end{enumerate}
\end{proposition}

\begin{proof} Given that DOAG is o-minimal, the equivalence between (1) and (2) follows by the Marker-Steinhorn theorem. For (1) implies (3), let $C_2$ be the convex hull of $G$ in $H$ and $C_1$ the biggest convex subgroup of $H$ such that all positive elements are below $G^{>0}$ (where  $G^{>0}:=\{x\in G: x>0\}$). By construction, $G$ embeds in $C_2/C_1$ and $C_2/C_1$ does not realize over $G$ any type of the form $\pm\infty$ nor $0^\pm$, therefore no type of the form $a^\pm$ for $a\in G$. Therefore, if $C_2/C_1\neq G$, $C_2/C_1$ must realize a cut over $G$ which implies that $C_2$ and hence $H$ realizes a cut over $G$. Thus $G=C_2/C_1$, hence $C_2=G\times C_1$ as groups. Now $C_1\trianglelefteq C_2$ implies that $C_2=G\times C_1$ as ordered groups. To show (3) implies (4), let $x\in H$ be such that $|x|<|y|$ for some $y\in G$. Then $x\in C_2$, so $c=g+e$ with $g\in G$ and $e\in C_1$. Finally, to show (4) implies (1), suppose that $a\in H$ lies in a cut over $G$. Then there is $y\in G$ such that $|a|<|y|$, but there is no $g\in G$ such that $|a-g|<G^{>0}$. \end{proof}

\begin{definition} Let $(K\subseteq L,v)$ be an extension of valued fields.
\begin{enumerate}
	\item A sequence $y_1,\ldots,y_n\in L$ is \emph{separated} if for all $k_1,\ldots,k_n\in K$ it satisfies
	$$v\left(\sum_{i=1}^n k_i y_i\right)=\min\{v(k_iy_i);1\leq i\leq n\}.$$
	\item The extension $(K\subseteq L,v)$ is \emph{separated} when every finitely generated $K$-vector subspace of $L$ has a separated basis, \textit{i.e.}., for all $x_1,\ldots,x_n\in L$ there are $y_1,\ldots,y_m\in L$ such that $\sum x_iK=\sum y_iK$ and $y_1,\ldots,y_m$ is a separated sequence. \end{enumerate}
	
\end{definition}

This definition was introduced by Baur in \cite{baur:82}. The following results about separated extensions will be later used (their proofs can be found in \cite{baur:82} and \cite{delon:88} respectively):

\begin{theorem}[Baur]\label{baur} If a valued field $K$ is maximal, then any valued field extension of $K$ is separated.
\end{theorem}

\begin{theorem}[Delon]\label{delon2} Let $(K\subseteq L,v)$ be an extension of algebraically closed valued fields.\footnote{Here $(K,v)$ algebraically maximal and Kaplansky conditions are enough.} Then the following are equivalent:
\begin{enumerate}
	\item the extension $(K\subseteq L,v)$ is separated; 
	\item $L$ is linearly disjoint over $K$ with any immediate extension of $K$ (in any common valued extension);
	\item for all $n$ and all $x,y_1,\ldots,y_n \in L$, $I_n(x,y_1,\ldots,y_n;K)$ has a maximal element.
\end{enumerate}
\end{theorem}

For a field $F$ we use $F^a$ to denote its field-theoretic algebraic closure. A simple but important corollary of the previous theorem is

\begin{corollary}\label{deloncoro}
Let $(K\subseteq L,v)$ be models of $ACVF$ and $m$ a trivially valued algebraically closed subfield of $L$. Then the extension $(K\subseteq K(m)^a,v)$ is separated.  	
\end{corollary}

\begin{lemma}\label{sepquotient} Let $(K\subseteq L,v)$ be an extension of valued fields. Then the following are equivalent:
\begin{enumerate}
	\item the extension is separated;
	\item for some (any) valuation $w\leq v$, both extensions $(K\subseteq L,w)$ and $(K/w\subseteq L/w, v/w)$ are separated.
\end{enumerate}
\end{lemma}

\begin{proof} We show $1\Rightarrow 2$, so let $w$ be a valuation such that $w\leq v$. The extension $(K\subseteq L,w)$ is $w$-separated: given that $w\leq v$, a $v$-separated sequence must be $w$-separated. So it remains to show that $(K/w\subseteq L/w, v/w)$ is separated. Take $a=(a_1,\ldots,a_n)\in (L/w)^n$. We show there is a $v/w$-separated basis for the $K/w$-vector subspace generated by $a$. Without loss of generality, we may assume that $a_1,\ldots,a_n$ are linearly independent over $K/w$. Therefore, any $b=(b_1,\ldots,b_n)\in L^n$ such that for all $1\leq i\leq n$, $b_i$ satisfies both that $w(b_i)=0$ and $a_i=b_i/w$, is a $w$-separated sequence (and also linearly independent over $K$). Let $b'$ be a $v$-separated basis for the $K$-vector subspace generated by $b$. Thus for each $1\leq j\leq n$, $b_j'=\sum k_{ji} b_i$ where $k_{ji}\in K$ for all $1\leq i\leq n$. Since $b$ is $w$-separated, $w(b_j')=w(k_{ji}b_i)$ for some $1\leq i\leq n$, so modulo some multiplication by an element in $K$, we can assume that $w(b_j')=0$ for all $1\leq j\leq n$ (note that this new sequence is also a $v$-separated basis for the $K$-vector space generated by $b$). It is easy to see that $b'/w$ is $v/w$-separated and clearly a basis for $a$. For the converse, \textit{i.e.}, $2\Rightarrow 1$, let $a=(a_1,\ldots,a_n)\in L^n$ and let $b=(b_1,\ldots,b_r)$ be a $w$-separated basis of the $K$-vector subspace generated by $a$. Let $\{v(b_i):1\leq i\leq r\}=\{\gamma_1,\ldots,\gamma_s\}$ where $\gamma_1<\cdots<\gamma_s$. For $1\leq j\leq s$, let $m_j$, $1\leq m_j\leq r$, be such that $v(b_{m_j})=\gamma_j$ and suppose, by renaming the sequence $b$, that for all $1\leq j\leq s$, all $b_m$ such that $m_j\leq m< m_{j+1}$ have the same value $\gamma_j$. Then for each $1\leq j\leq s$ define the sequence
$$\beta_j:=(b_i b_{m_j}^{-1}/w: m_j\leq i< m_{j+1}).$$
For each $1\leq j\leq s$, there is a $v/w$-separated basis $\alpha_j$ for the $K/w$-vector subspace generated by $\beta_j$. For $a_j'\in \sum b_i b_{m_j}^{-1} K$ an arbitrary lifting of $\alpha_j$, $(b_{m_j}a_j)_{1\leq j\leq s}$ constitutes a $v$-separated basis for the $K$-vector subspace generated by $a$. \end{proof}

We are now ready to prove theorem \ref{charACVF}:

\begin{proof}[of theorem \ref{charACVF}:] The proof idea of $2\Rightarrow 1$ goes back to Delon in \cite{delon:89}; we include it here for completion. By \ref{general1}, $T_n(vK,vL)$ holds for all $n\geq 1$. In order to show $T_n(K,L)$ for all $n\geq 1$, we show instead that $W_n(K,L)$ holds for all $n\geq 1$. By elimination of quantifiers in $ACVF$ in the language of rings together with a predicate for the divisibility relation ($x|y$ if and only if $v(x)\leq v(y)$), it is enough to consider definable subsets of $L$ defined by formulas of the form
\begin{equation}\label{eq:1}
v(P(x))\geq v(Q(x)),
\end{equation}
where $P,Q\in L[X]$ for $X:=(X_1,\ldots,X_n)$. The $K$-vector subspace of $L$ generated by the coefficients in $P$ and $Q$ has a separated basis $\{l_1,\ldots,l_m$\}, which means that we can rewrite $P$ and $Q$ as
	$$P(X)= \Sigma l_ip_i, \text{ with $p_i\in K[X]$}	\hspace{1cm} Q(X)= \Sigma l_iq_i, \text{ with $q_i\in K[X]$}$$
where for $x\in K$ we have
$$v(P(x))=\min\{v(l_ip_i(x));i\}	\hspace{1cm} v(Q(x))=\min\{v(l_iq_i(x));i\}.$$

\noindent Therefore (\ref{eq:1}) is equivalent to
\small
\begin{equation*}
\bigvee_{i_P, l_Q}\left(\bigwedge_i \left(v(l_{i_P}p_{i_P}(x))\leq v(l_ip_i(x)) \wedge v(l_{i_Q}q_{i_Q}(x))\leq v(l_iq_i(x))\right) \wedge v(l_{i_Q}q_{i_Q}(x))\leq v(l_{i_P}p_{i_P}(x))\right).
\label{eq:2}
\end{equation*}
\normalsize
Each inequality of the form

$$v(l'q(x))\leq v(lp(x)),$$
where $l,l'\in L$ and $p,q\in K[X]$ is equivalent to

\begin{equation}
v(l')+v(q(x))\leq v(l)+v(p(x)).
\label{eq:3}
\end{equation}

\noindent By $W_n(vK,vL)$, there is a formula $\phi(u,v)$ with parameters in $K$ such that formula $\phi(v(p(x)),v(q(x)))$ is equivalent to
(\ref{eq:3}). Therefore, the set of points in $K$ satisfying the formula (\ref{eq:1}) is $K$-definable.

\

To show $1\Rightarrow 2$, suppose for a contradiction that $(K\subseteq L,v)$ is an extension such that $T_n(vK,vL)$ holds for all $n$ but the extension is not separated. Also, by Corollary \ref{T1} no point in $L$ is limit over $K$.

\begin{claim}\label{cla:2.1} We can assume that $vK=vL$.
\end{claim}

By Proposition \ref{doag}, there are convex subgroups $C_1\trianglelefteq C_2\trianglelefteq vL$ such that $C_2=vK\times C_1$.
Consider on $L$ the valuation $v_1:=v/C_1$. It satisfies that $v_1\upharpoonright K= v\upharpoonright K$. This implies that $(v/v_1)(K/v_1)=0$ so the extension $(K/v_1\subseteq L/v_1,v/v_1)$ is separated (note that if $v\upharpoonright K=0$ then $(K,v)$ is maximal, which by Theorem \ref{baur} implies that every extension if separated). By lemma \ref{sepquotient}, the extension $(K\subseteq L,v_1)$ is not separated. Furthermore $v_1K$ is convex in $v_1L$. Thus we can define $v_2=v_1/v_1K$. We have that $v_2 \upharpoonright K=0$, therefore $(K\subseteq L,v_2)$ is separated. Applying lemma \ref{sepquotient} again, $(K/v_2\subseteq L/v_2,v_1/v_2)$ is not separated. But $(v_1/v_2)(K/v_2)=(v_1/v_2)(L/v_2)=vK$, therefore $L/v_2$ is for the valuation $v_1/v_2$ a non-separated extension of $K/v_2$ with the same value group. Furthermore, $L/v_2$ contains no limit point over $K/v_2$ with respect to $v_1/v_2$ as any $l\in L$ such that $l/v_2$ is limit over $K/v_2$ with respect to $v_1/v_2$ will be \emph{a fortiori} limit over $K$ with respect to $v$. This completes the claim.

\

We assume now $vL=vK$. By Theorem \ref{delon2}, let $n$ be a positive integer minimal such that there are $x\in L$ and $y=(y_1,\ldots,y_n)\in L^n$ such that $I_n(x,y;K)$ has no maximal element. Notice that we can assume $n\geq 1$ given that $L$ does not contain any limit point over $K$. Moreover, since $vK=vL$, we can also assume that $v(x)=v(y_i)=0$ for all $1\leq i\leq n$ (notice that for all $e,e_1,\ldots,e_n\in K$, $I_n(ex,(e_1y_1,\ldots,e_ny_n);K)=I_n(x,y;K)+v(e)$).

\begin{claim}\label{cla:2.3} We can assume that $y_1/v,\ldots,y_n/v$ are linearly independent over $K/v$.
\end{claim}

For $m\leq n$ and $z_1,\ldots,z_m\in L$ we let $\gamma_{m-1}(z_m,(z_1,\ldots,z_{m-1}))$ be the maximal element in $I_{m-1}(z_m,(z_1,\ldots,z_{m-1});K)$ which exists by the minimality of $n$. We show by induction on $m$, with $1\leq m\leq n$, that every $K$-vector subspace of $L$ generated by less than $m$ elements (that we suppose linearly independent over $K$) has a basis $(y_1,\ldots,y_m)$ such that $y_1/v,\ldots,y_m/v$ are linearly independent over $K/v$. By induction, we can assume that $y_1/v,\ldots,v_{m-1}/v$ are linearly independent over $K/v$ (note that $m=1$ is possible here). Since $vK=vL$ we have that $\gamma_{m-1}(y_m,(y_1,\ldots,y_{m-1}))=v(e)$ for some $e\in K$. For $1\leq i\leq m-1$, let $e_i\in K$ be such that $v(y_m+\sum_{i=1}^{m-1}e_iy_i)=v(e)$ and $y_m':=(y_m+\sum_{i=1}^{m-1}e_iy_i)e^{-1}$. If there were $a_1,\ldots,a_m \in K$ such that $v(a_my_m'+\sum_{i=1}^{m-1}a_iy_i)>0$ and $v(a_i)=0$ for all $1\leq i\leq m$, then

\begin{center}
	\begin{tabular}{lll}
$v(a_my_m'+\sum_{i=1}^{m-1}a_iy_i)>0 $ &$\Leftrightarrow$ & $v(y_m'+\sum_{i=1}^{m-1}a_ia_m^{-1}y_i)>0$\\
																			&$\Leftrightarrow$ & $v(((y_m+\sum_{i=1}^{m-1}e_iy_i)e^{-1})+\sum_{i=1}^{m-1}a_ia_m^{-1}y_i)>0$\\
																			&$\Leftrightarrow$ & $v(y_m+(\sum_{i=1}^{m-1}e_iy_i+\sum_{i=1}^{m-1}a_ia_m^{-1}ey_i))>v(e)$,\\

\end{tabular}
\end{center}

\noindent which contradicts the maximality of $v(e)$ in $I_{m-1}(y_m,(y_1,\ldots,y_{m-1});K)$. Therefore elements
$y_1/v,\ldots,y_{m-1}/v,y_m'/v$ are linearly independent over $K/v$, which completes the proof of the claim.

\

Let $(\gamma_\alpha)_{\alpha<\alpha_0}\subseteq I_n(x,y;K)$ be a strictly increasing cofinal sequence in $I_n(x,y;K)$. For $1\leq i\leq n$ let $(k_{i,\alpha})_{\alpha<\alpha_0}$ be sequences such that for each $\alpha<\alpha_0$, $k_{i,\alpha}\in K$ and $\gamma_\alpha=v(x+\sum y_ik_{i,\alpha})$.

\begin{claim}\label{cla:2.4} We can assume that for each $1\leq i\leq n$, the sequence $(k_{i,\alpha})_{\alpha<\alpha_0}$ is a PC sequence (pseudo-Cauchy sequence) with $v(k_{i,\alpha}-k_{i,\beta})=\gamma_\alpha$ whenever $\alpha<\beta<\alpha_0$.
\end{claim}

By Claim \ref{cla:2.3} we can assume that $y_1,\ldots,y_n$ is a separated basis with $v(y_i)=0$ for any $1\leq i\leq n$. Therefore, for $\alpha<\beta<\alpha_0$, given that $v(\sum y_i(k_{i,\alpha}-k_{i,\beta}))=\gamma_\alpha$, by separation there is $1\leq i\leq n$ such that $\gamma_\alpha=v(y_i)+v(k_{i,\alpha}-k_{i,\beta})=v(k_{i,\alpha}-k_{i,\beta})$. As a consequence, there must be $1\leq j\leq n$ such that, for $\alpha$ in a cofinal subset of $\alpha_0$ and any $\beta>\alpha$, we have $\gamma_\alpha=v(k_{j,\alpha}-k_{j,\beta})$. So $(k_{j_\alpha})_{\alpha<\alpha_0}$ is a PC sequence. Furthermore, for any $1\leq i\leq n$ we have that $v((k_{i,\alpha}-k_{i,\beta}) y_i)\geq \gamma_\alpha$ by definition of separated basis. By the minimality of $n$ we have in fact equality for $\alpha$ cofinal in $\alpha_0$. For suppose there is $1\leq l\leq n$ and some $\alpha_1<\alpha_0$ such that for all $\alpha_1\leq\alpha<\beta<\alpha_0$, $v((k_{l,\alpha}-k_{l,\beta}) y_l)> \gamma_{\alpha}$. Then for $\alpha_1\leq \alpha<\alpha_0$ we will also have 
$$v(x+\sum_{i=1}^n k_{i,\alpha}y_i)=\gamma_\alpha=v(x+k_{l,\alpha_1}y_l + \sum_{i\neq l}k_{i,\alpha}y_i),$$
so taking $x'=x+k_{l,\alpha_1}y_l$ will get that $\gamma_\alpha \in I_{n-1}(x',(y_i)_{i\neq l};K)$ for all $\alpha_1<\alpha<\alpha_0$, which contradicts the minimality of $n$ since clearly $I_{n-1}(x',(y_i)_{i\neq l};K)\subseteq I_n(x,y;K)$. Therefore each sequence $(k_{i,\alpha})_{\alpha<\alpha_0}$ is PC.

\begin{claim}\label{cla:2.5} For all $1\leq i\leq n$, the sequence $(k_{i,\alpha})_{\alpha<\alpha_0}$ does not have a pseudo-limit in $K$.
\end{claim}

Suppose for a contradiction there is $1\leq i\leq n$ such that $(k_{i,\alpha})_{\alpha<\alpha_0}$ has a pseudo-limit $k\in K$. Let $y'=(y_1,\ldots,y_{i-1},y_{i+1},\ldots,y_n)$. Given that for each $\alpha<\alpha_0$, $v(k_{i,\alpha}-k)=\gamma_\alpha$, it is not difficult to show that $I_{n-1}(x-ky_i,y';K)=I_n(x,y;K)$, which contradicts the minimality of $n$.

\

Let $K^\ast$ be a $|K|^+$-saturated extension of $K$ and for each $1\leq i\leq n$ let $k_i^\ast\in K^\ast$ be a pseudo-limit for $(k_{i,\alpha})_{\alpha<\alpha_0}$. By construction, we have that for all $e_1,\ldots,e_n\in K^n$ and $\alpha<\alpha_0$
$$v\left(x+\sum_{i=1}^n e_iy_i\right)>\gamma_\alpha \Leftrightarrow \bigwedge_{i=1}^n v(e_i-k_i^\ast)>\gamma_\alpha.$$

\noindent We show that the set $D:=\{(\gamma,k_1,\ldots,k_n)\in vK\times K^n: v(x+\sum k_iy_i)\geq \gamma\}$ is not definable in $K$, which implies that $T_{n+1}(K,L)$ does not hold. Indeed, if it was, the projections on $vK\times K$ (where $K$ is say the first copy in $K^n$) will also be definable. Given that $ACVF$ is definably maximal, a pseudo-limit of $(k_{1,\alpha})_{\alpha<\alpha_0}$ would belong to $K$ contradicting Claim \ref{cla:2.5}. \end{proof}

\section{Counterexample}

As stated in the introduction, this section is devoted to Proposition \ref{counterexample}. 

\

\begin{proof}[Proof of Proposition \ref{counterexample}:] The idea is to build an extension $K\prec L$ of algebraically closed valued fields with the following properties:
\begin{enumerate}[$\cdot$]
	\item $td(L;K)=n+1$
	\item The extension $(K\subseteq L,v)$ is not separated
	\item For any $M$ such that $K\subseteq M\subseteq L$ and $td(M;K)\leq n$, the extension $(K\subseteq M, v)$ is separated.
\end{enumerate}

Let $k$ be an algebraically closed field and $K=k(X,Y)^a$ with $0=v(k)<v(X)<<v(Y)$. Consider elements in a big valued field extension of $K$ with the following properties:

\begin{enumerate}
		\item $\bar{t}:=(t_1,\ldots, t_n)$ such that $v(k(t_1,\ldots,t_n))=0$ and $t_1/v,\ldots,t_n/v$ are algebraically independent over $k$;
		\item $\bar{f}:=(f_1,\ldots,f_{n+1})$, with $f_1,\ldots,f_{n+1}\in k((X))$ algebraically independent over $k(X)$, and $v(f_i)=0$ for all $1\leq i\leq n+1$;
		\item $t_{n+1}:= \sum_{i=1}^n f_it_i +f_{n+1}$.
\end{enumerate}

\noindent Note that $f_1,\ldots,f_{n+1}$ are still algebraically independent over $k(\bar{t},X)$. Indeed, $k((X))$ and $k(\bar{t},X)$ are linearly disjoint over $k(X)$ by Proposition \ref{delon2} $(1\Rightarrow 2)$. Finally, let $L:=K(t_1,\ldots,t_{n+1})^a$. By construction, $k(\bar{t},X,Y)^a\subseteq L\subseteq k(\bar{t})((X))(Y)^a$, hence $vK=vL$ and $L/v=k(t_1/v,\ldots,t_n/v)^a$. The result is a consequence of the following two claims:

\begin{claim}\label{claim1} $\mathbb{Z}v(X)$ is cofinal in $I_{n+1}(t_{n+1}, 1,t_1,\ldots,t_n;K)$.
\end{claim}

\begin{claim}\label{claim2} For any field $M$ such that $K\subseteq M\subset L$ and $td(M,K)\leq n$, $M$ is a separated extension of $K$.
\end{claim}

Suppose for the sake of argumentation both claims are true. By theorem \ref{charACVF}, Claim \ref{claim2} shows $T_i(K,L)$ for all $i\leq n$. Claim \ref{claim1} shows that $T_{n+1}(K,L)$ is not true since the definability of the type $tp(t_1,\ldots,t_{n+1}/K)$ would imply the definability of $I_{n+1}(t_{n+1}, 1,t_1,\ldots,t_n;K)$, which defines a cut in $vK$. We proceed to prove the claims.

\

\emph{Proof of Claim \ref{claim1}:} We have that 
\small
$$I_{n+1}(t_{n+1}, 1,t_1,\ldots,t_n;K)=\left\{v\left(f_{n+1}+e_{n+1}+\sum_{i=1}^n t_i(f_i+e_i)\right): e_1,\ldots,e_{n+1}\in K\right\}.$$
\normalsize
Now $v(t_i)=0$ for all $1\leq i\leq n$, $f_i\in k((X))$ for all $1\leq i\leq n+1$ and $K$ contains $k(X)$. This shows $\mathbb{Z}v(X)\subseteq I_{n+1}(t_{n+1}, 1,t_1,\ldots,t_n;K)$. It remains then to show $I_{n+1}(t_{n+1},1,t_1,\ldots,t_n;K)\subseteq C$ where $C\trianglelefteq L$ is the biggest convex subgroup of $\mathbb{Q}\times\mathbb{Q}$ not containing $v(Y)$. Suppose towards a contradiction that there are $e_1\ldots,e_{n+1}\in K$ such that for $v(t_{n+1}+e_{n+1}+\sum_{i=1}^n e_it_i)>C$. Thus, if we introduce the valuation $w:=v/C$, we have
$$\left(\sum_{i=1}^n f_it_i +f_{n+1}+e_{n+1}+ \sum_{i=1}^n e_it_i\right)/w= 0$$
which contradicts that $f_1,\ldots,f_{n+1}$ are algebraically independent over $k(\bar{t},X)$ given that $K/w=k(X)^a\subseteq k((X))^a$. This completes the Claim \ref{claim1}.

\

\emph{Proof of Claim \ref{claim2}:} We have the following diagram: 

\begin{center}
	\includegraphics{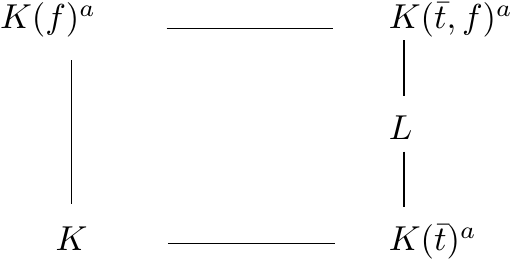}
\end{center}	

\noindent By Corollary \ref{deloncoro}, the extension $K\subseteq K(\bar{t})^a$ is separated and by Theorem \ref{delon2} $K(\bar{t})^a$ and $K(\bar{f})^a$ are linearly disjoint over $K$. Suppose now there exists some algebraically closed field $M$ such that $K\subseteq M\subset L$ (hence $td(M,K)\leq n$) and the extension $(K\subseteq M, v)$ is not separated. Let $m$ be a lifting of $M/v$ in $M$ and $l$ a lifting of $L/v$ in $L$ such that $k\subseteq m\subseteq l$. Given that 

\begin{center}
	\begin{tabular}{lll}
$n$						&$=$& $td(l;k)$\\
							&$=$& $td(l;m)+td(m;k)$\\
							&$\leq$& $td(L;M)+td(K(m)^a;K)$\\
							&$=$& $td(L;K)-td(M;K(m)^a)$\\
							&$\leq$& $n+1-td(M;K(m)^a)$,\\
\end{tabular}
\end{center}

\noindent we have that $td(M;K(m)^a)\leq 1$. If $td(M;K(m)^a)=0$, then $M=K(m)^a$ and $(K\subseteq M,v)$ is separated by Corollary \ref{deloncoro}. Suppose towards a contradiction that $td(M;K(m)^a)=1$. This implies that $L=M(l)^a$. Therefore we have the following diagram: 

\begin{center}
	\includegraphics{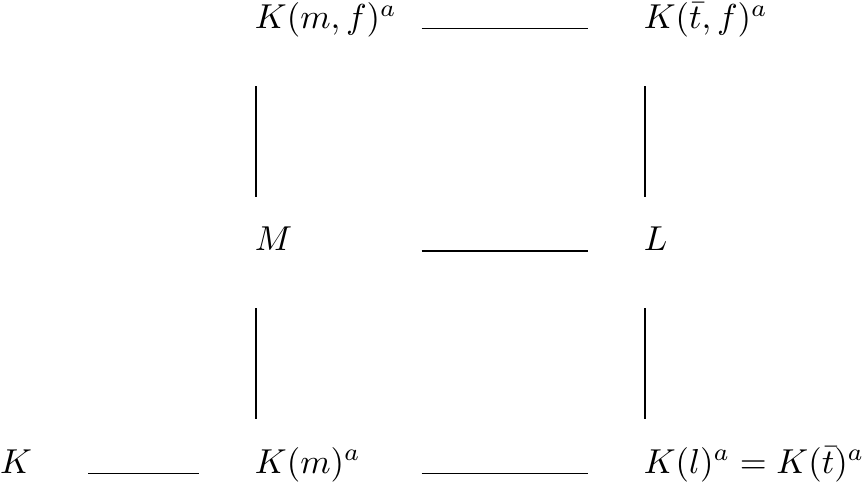}
\end{center}	

\noindent By Corollary \ref{deloncoro}, since $L=M(l)^a$, the extension $M\subseteq L$ is separated. Notice that by construction, $M\subseteq K(m,\bar{f})^a$ is an immediate extension. Therefore, by implication $(1\Rightarrow 1)$ of Theorem \ref{delon2}, $L$ and $K(m,\bar{f})^a$ are linearly disjoint over $M$. But this implies that the field of definition of the ideal
$$Ann(t_1,\ldots,t_n,t_{n+1};K(m,\bar{f})^a):=\{P\in K(m,\bar{f})^a[T_1,\ldots,T_n,X]: P(t_1,\ldots,t_{n+1})=0\}$$
is contained in $M$ (see for instance chapter III, Theorem 8 in \cite{lang:58}). On the other hand, this ideal is generated by the polynomial
$$X-\left(\sum_{i=1}^n T_if_i+f_{n+1}\right)$$
which implies that $f_1,\ldots,f_{n+1}$ belong to the field of definition. Since they are algebraically independent over $K$, $td(K(\bar{f});K)=n+1$ contradicts that $td(M;K)\leq n$. 

\

This construction provides valued fields of equal characteristic. To get the counterexample for models of $ACVF_{0,p}$ with $p$ a prime number, take for $k$ the algebraic closure of $\mathbb{Q}_p$ and replace $v$ by its composition with the $p$-adic valuation. \end{proof}

While the previous result shows that the Marker-Steinhorn theorem cannot be generalized to $ACVF$ and therefore cannot be generalized to $C$-minimality, the following remains true:

\begin{theorem}\label{absolute}
Let $K$ be a model of $ACVF$. Then, $T_1(K)$ implies $T_n(K)$ for all $n\geq 1$.
\end{theorem}

\begin{proof} Let $K$ be a model of $ACVF$ such that $T_1(K)$ holds. By \ref{T1}, we have that $T_1(vK,vL)$ and for all extensions $L$ of $K$, $L$ does not contain limit points over $K$. Therefore, $K$ is in particular maximal. By Baur's Theorem \ref{baur}, this implies that all extensions are separated. The result follow by \ref{charACVF}. \end{proof}	

Question \ref{absoluea} remains open. We do not know either of any counterexample to question \ref{absoluea} for weakly o-minimal theories. 

\bibliographystyle{amsplain}
\bibliography{merged}

\end{document}